\documentclass{amsart}

\usepackage{amssymb,amsmath,amsthm,xypic,amsfonts,latexsym,booktabs,tikz}

\usetikzlibrary{calc}

\theoremstyle{definition}
\newtheorem{definition}{Definition}[section]

\newtheorem{remark}[definition]{Remark}

\theoremstyle{plain}

\newtheorem{proposition}[definition]{Proposition}
\newtheorem{theorem}[definition]{Theorem}

\newtheorem{conjecture}[definition]{Conjecture}

\newcommand{\dend}{\mathsf{Dend}}
\newcommand{\dias}{\mathsf{Dias}}

\newcommand{\non}{\mathrm{Non}}
\newcommand{\den}{\mathrm{Den}}

\allowdisplaybreaks

\begin{document}

\title[Special identities for pre-Jordan algebras]
{Special identities for the pre-Jordan product in the free dendriform algebra}

\author[Bremner]{Murray R. Bremner}

\address{Department of Mathematics and Statistics,
University of Saskatchewan, Canada}

\email{bremner@math.usask.ca}

\author[Madariaga]{Sara Madariaga}

\address{Department of Mathematics and Statistics,
University of Saskatchewan, Canada}

\email{madariaga@math.usask.ca}

\dedicatory{In memory of Jean-Louis Loday (1946--2012)}



\begin{abstract}
Pre-Jordan algebras were introduced recently in analogy with pre-Lie algebras.
A pre-Jordan algebra is a vector space $A$ with a bilinear multiplication $x \cdot y$ such that
the product $ x \circ y = x \cdot y + y \cdot x $ endows $A$ with the structure of a Jordan algebra,
and the left multiplications $L_\cdot(x)\colon y \mapsto x \cdot y$ define a representation
of this Jordan algebra on $A$.
Equivalently, $x \cdot y$ satisfies these multilinear identities:
  \begin{align*}
  &
  (x \circ y) \cdot (z \cdot u) + (y \circ z) \cdot (x \cdot u) + (z \circ x) \cdot (y \cdot u)
  \\
  &\quad
  \equiv
  z \cdot [(x \circ y) \cdot u] + x \cdot [(y \circ z) \cdot u] + y \cdot [(z \circ x) \cdot u],
  \\[2pt]
  &
  x \cdot [y \cdot (z \cdot u)] + z \cdot [y \cdot (x \cdot u)] + [(x \circ z) \circ y] \cdot u
  \\
  &\quad
  \equiv
  z \cdot [(x \circ y) \cdot u] + x \cdot [(y \circ z) \cdot u] + y \cdot [(z \circ x) \cdot u].
  \end{align*}
The pre-Jordan product $x \cdot y = x \succ  y + y  \prec x$ in any dendriform algebra also 
satisfies these identities.
We use computational linear algebra based on the representation theory of the symmetric group
to show that every identity of degree $\le 7$ for this product is implied
by the identities of degree 4, but that there exist new identities of degree 8 which do not follow
from those of lower degree.
There is an isomorphism of $S_8$-modules between these new identities and the special identities
for the Jordan diproduct in an associative dialgebra.
\end{abstract}

\maketitle


\section{Introduction}

During the last few decades many new algebraic structures have been discovered,
connecting several areas in mathematics and physics.
In particular, Loday \cite{Loday01} defined a dendriform algebra to be a vector space
$A$ with two bilinear operations $\succ$, $\prec$ satisfying these identities:
\begin{align*}
	( x \prec y ) \prec z &\equiv x \prec ( y \prec z ) + x \prec ( y \succ z ), \\
	( x \succ y ) \prec z &\equiv x \succ ( y \prec z ), \\
	x \succ ( y \succ z ) &\equiv ( x \prec y ) \succ z + ( x \succ y ) \succ z.
\end{align*}
The operation $x * y = x \prec y + x \succ y$ is associative,
and so this new structure dichotomizes the notion of associative algebra, splitting
the associative product into a sum of two operations.
Dendriform algebras are also closely related to associative dialgebras, through the
framework of Koszul duality for algebraic operads.
The notion of associative dialgebra defines a binary quadratic operad $\dias$, and
the dual operad $\dias^!$ as defined by Ginzburg and Kapranov \cite{GinzburgKapranov}
is precisely the operad $\dend$ of dendriform algebras.
Further splitting of the operations gives rise to the quadri-algebras introduced by
Aguiar and Loday \cite{AguiarLoday04}.
For a general theory of splitting operations, see Bai et al.~\cite{BaiBellierGuoNi12} and
Vallette \cite{Vallette08}.
For an introduction to operads from the point of view of algebraic structures, see the lecture notes of Vatne \cite{Vatne04}.
For a comprehensive monograph on algebraic operads, see Loday and Vallette \cite{LodayVallette12}.

\subsection*{Pre-Lie and pre-Jordan algebras}

The operation $x \cdot y = x \succ y - y \prec x$ in a dendriform algebra
satisfies the defining identities for pre-Lie algebras, which had
already been introduced independently in several areas of mathematics and then appeared in
other fields; see \cite{Burde06} and the references therein.
Pre-Lie algebras can be regarded as the algebraic structures behind the classical
Yang-Baxter equation, which plays an important role in integrable systems and quantum groups.
The L-dendriform algebras introduced by Bai, Liu and Ni \cite{BaiLiuNi10}
are related to pre-Lie algebras in the same way that pre-Lie algebras are related to Lie algebras.

We remark that the definition of dendriform algebra can be motivated in terms of pre-Lie algebras.  
Starting with the pre-Lie identity,
  \[
  ( x \cdot y ) \cdot z - x \cdot ( y \cdot z )
  \equiv
  ( y \cdot x ) \cdot z - y \cdot ( x \cdot z ),
  \]
expanding it using the preceding operation in a dendriform algebra,
and identifying terms with the same permutation of the variables gives these identities,
  \begin{align*}
    ( x \succ y ) \succ z
  - x \succ ( y \succ z )
  &\equiv
  - ( x \prec y ) \succ z,
  \\
  x \succ ( z \prec y )
  &\equiv
  ( x \succ z ) \prec y,
  \\
  - ( y \prec x ) \succ z
  &\equiv
    ( y \succ x ) \succ z
  - y \succ ( x \succ z ),
  \\
  ( y \succ z ) \prec x
  &\equiv
  y \succ ( z \prec x ),
  \\
  - z \prec ( x \succ y )
  &\equiv
    z \prec ( x \prec y )
  - ( z \prec x ) \prec y,
  \\
    z \prec ( y \prec x )
  - ( z \prec y ) \prec x
  &\equiv
  - z \prec ( y \succ x ),
  \end{align*}
which are equivalent to the identities defining dendriform algebras.

Pre-Jordan algebras as Jordan analogues of pre-Lie algebras were introduced
by Hou, Ni and Bai \cite{HouNiBai}.
Explicitly,
\begin{itemize}
	\item
	if $(A, \cdot)$ is a pre-Lie (respectively pre-Jordan) algebra then the product
	$[x,y] = x \cdot y - y \cdot x$ (resp. $x \circ y = x \cdot y + y \cdot x$) defines
	a Lie (resp. Jordan) algebra,
	\item
	a vector space with a bilinear product $(A,\cdot)$ is a pre-Lie (respectively pre-Jordan) algebra
	if and only if $(A,[-,-])$ (resp. $(A, \circ)$) defined above is a Lie (resp. Jordan) algebra and
	$(L_\cdot,A)$ is a representation of $(A,[-,-])$ (resp. $(A, \circ)$).
\end{itemize}
There are analogues of the classical Yang-Baxter equation in Jordan algebras
\cite{Zhelyabin00} and pre-Jordan algebras \cite{HouNiBai}.
The J-dendriform algebras introduced by Bai and Hou \cite{BaiHou12}
are related to pre-Jordan algebras in the same way that pre-Jordan algebras are related to Jordan algebras.

These structures are interconnected as shown in the commutative diagram of categories in Table \ref{BaiHouTable},
which has been described in detail by Bai and Hou \cite{BaiHou12}.
We are most interested in the vertical arrows, which use analogues of Lie and Jordan products
on an algebra in one variety to endow the underlying vector space with the structure
of an algebra in another variety.

\begin{table}
\small
\begin{tabular}{|c|}
\hline
\\[-6pt]
\begin{tikzpicture}
\node[draw=none,fill=none] at (0,1.25) (Lie) {Lie algebra};
\node[draw=none,fill=none] at (3,1.25) (preL) {\begin{tabular}{c}pre-Lie\\algebra\end{tabular}};
\node[draw=none,fill=none] at (6,1.25) (Ldend) {\begin{tabular}{c}L-dendriform\\algebra\end{tabular}};
\node[draw=none,fill=none] at (3,0) (ass) {\begin{tabular}{c}associative\\algebra\end{tabular}};
\node[draw=none,fill=none] at (6,0) (dend) {\begin{tabular}{c}dendriform\\algebra\end{tabular}};
\node[draw=none,fill=none] at (9,0) (quad) {quadrialgebra};
\node[draw=none,fill=none] at (0,-1.25) (Jord) {Jordan algebra};
\node[draw=none,fill=none] at (3,-1.25) (preJ) {\begin{tabular}{c}pre-Jordan\\algebra\end{tabular}};
\node[draw=none,fill=none] at (6,-1.25) (Jdend) {\begin{tabular}{c}J-dendriform\\algebra\end{tabular}};
\draw [->] (Ldend.west) -- (preL.east);
\draw [->] (preL.west) -- (Lie.east);
\draw [->] (quad.west) -- (dend.east);
\draw [->] (dend.west) -- (ass.east);
\draw [->] (Jdend.west) -- (preJ.east);
\draw [->] (preJ.west) -- (Jord.east);
\draw [->] (quad.west) -- (Ldend.east);
\draw [->] (dend.west) -- (preL.south east);
\draw [->] (ass.west) -- (Lie.south);
\draw [->] (quad.west) -- (Jdend.east);
\draw [->] (dend.west) -- (preJ.north east);
\draw [->] (ass.west) -- (Jord.north);
\end{tikzpicture}
\\
\hline
\end{tabular}
\medskip
\caption{Commutative diagram of categories of algebras}
\label{BaiHouTable}
\end{table}

Pre-Lie algebras can be regarded as the algebraic structures behind the classical Yang-Baxter equation (CYBE),
which plays an important role in integrable systems and quantum groups.
This can be seen more clearly in terms of $\mathcal{O}$-operators of a Lie algebra introduced 
by Kupershmidt \cite{Kupershmidt99}
who generalized the operator form of the CYBE in a Lie algebra.
Moreover, the CYBE and the Rota-Baxter operators on a Lie algebra are understood as the special
$\mathcal{O}$-operators corresponding to the co-adjoint representation and the adjoint representation.
In terms of $\mathcal{O}$-operators, there are analogues of the CYBE in Jordan algebras
\cite{Zhelyabin00} and pre-Jordan algebras \cite{HouNiBai}.
The arrows in Table \ref{BaiHouTable} can be reversed using Rota-Baxter operators
\cite{AguiarLoday04,BaiHou12,BaiLiuNi10}:
\begin{itemize}
	\item a Rota-Baxter operator on an associative algebra gives a dendriform algebra,
	\item a Rota-Baxter operator on a dendriform algebra, or two commuting Rota-Baxter operators
		 on an associative algebra, gives a quadri-algebra,
	\item a Rota-Baxter operator on a Lie algebra gives a pre-Lie algebra,
	\item a Rota-Baxter operator on a pre-Lie algebra, or two commuting Rota-Baxter operators
		on a Lie algebra, gives an L-dendriform algebra,
	\item a Rota-Baxter operator on a Jordan algebra gives a pre-Jordan algebra,
	\item a Rota-Baxter operator on a pre-Jordan algebra, or two commuting Rota-Baxter operators
		on a Jordan algebra, gives a J-dendriform algebra.
\end{itemize}

\subsection*{Polynomial identities for bilinear operations}

In the present paper we investigate polynomial identities satisfied by
the pre-Jordan product $ x \cdot y = x \succ y + y \prec x $ in the free dendriform algebra;
our approach relies heavily on computational linear algebra.
We first verify that there are no identities of degree 3;
we then consider degree 4 and show that the identities satisfied by this operation
are equivalent to the defining identities for pre-Jordan algebras.
We next show that every identity of degree $\le 7$ satisfied by this operation is
implied by the identities of degree 4.
We finally consider degree 8, and obtain special identities for the pre-Jordan product:
identities satisfied by this operation but not implied by the identities of degree 4.

Comparing these results with those of Bremner and Peresi \cite{BremnerPeresi11}
for Jordan dialgebras, we observe some remarkable facts.
The Jordan diproduct $x \dashv y + y \vdash x$ in an associative dialgebra satisfies
the right commutative identity in degree 3; but the pre-Jordan product in a
dendriform algebra satisfies no identities in degree 3.
For each $n \ge 4$, we obtain two modules over the symmetric group $S_n$:
  \begin{itemize}
  \item
The first $S_n$-module consists of the identities for the Jordan diproduct in degree $n$
which do not follow from identities of lower degree, where an identity is a
multilinear polynomial in the free right commutative algebra which expands to zero in the
free associative dialgebra when the right commutative operation is replaced by the Jordan diproduct.
  \item
The second $S_n$-module consists of the identities for the pre-Jordan product in degree $n$
which do not follow from identities of lower degree, where an identity is a
multilinear polynomial in the free nonassociative algebra
which expands to zero in the free dendriform algebra when the nonassociative operation
is replaced by the pre-Jordan product.
  \end{itemize}
We show that these $S_n$-modules are isomorphic for $4 \le n \le 8$.
This suggests a relation between the identities
satisfied by corresponding bilinear operations in the free algebras in two varieties
defined by dual operads.
We give a precise statement of this conjecture; see Section \ref{conjecturesection}.

\subsection*{Computational methods}

We conclude this introduction with a brief summary of the theoretical aspects of our computational approach.
For further information, see \cite{BremnerMurakamiShestakov,BremnerPeresi11,Hentzel77} and the references therein.

We say that a multilinear nonassociative polynomial $f(x_1, \dots, x_n)$ is an identity in an algebra $A$
if $f(a_1, \dots, a_n) = 0$ for any $a_1, \dots, a_n \in A$; we usually abbreviate this by
$f(x_1, \dots, x_n) \equiv 0$.
We regard the subspace of all identities of degree $n$ for an algebra $A$ as a module over the symmetric group
$S_n$ acting by permutations of the variables.
Given identities $f, f_1, \dots, f_k$ of degree $n$, we say that $f$ is a consequence of $f_1, \dots, f_k$
if $f$ belongs to the $S_n$-submodule generated by $f_1, \dots, f_k$;
in other words, $f$ is a linear combination of $f_1, \dots, f_k$, allowing permutations of the variables.

Let $\non(X)$ be the free nonassociative algebra on the generating set $X$ over the field $F$, and
let $A_n$ be the subspace of $\non(X)$ consisting of the multilinear polynomials of degree $n$.
Let $\den(X)$ be the free dendriform algebra on $X$ over $F$,
and let $B_n$ be the subspace of $\den(X)$ consisting of the multilinear polynomials of degree $n$.
For each $n$ we consider the linear map $\mathcal{E}_n \colon A_n \to B_n$, which we call the expansion map,
defined by replacing each occurrence of the nonassociative product $x \cdot y$ in $\non(X)$ by the
pre-Jordan product $x \succ y + y \prec x$ in $\den(X)$.
The basic computational principle is that the kernel of $\mathcal{E}_n$ consists of the identities of degree $n$
satisfied by the pre-Jordan product in the free dendriform algebra.
As $n$ grows, so does the size of the matrices, so different techniques are needed to compute this kernel,
as will be detailed below.

All the computations described in this paper were performed using Maple 16.
In general, we assume that the base field $F$ has characteristic 0, and use rational arithmetic for our
computations.
However, as the sizes of the matrices increase, it becomes impossible to compute row canonical forms
using rational arithmetic.
In higher degrees, we therefore use modular arithmetic with a suitable prime $p$.
To justify this, recall that the polynomial identities we consider are multilinear, 
so the spaces of identities in degree $n$
have a natural structure of a module over the symmetric group $S_n$.
The group algebra $F S_n$ is semisimple when the base field $F$ has characteristic 0 or $p > n$.
This implies that the ranks of the matrices will be the same whether we use rational arithmetic
or modular arithmetic with a prime greater than the degree of the identities.
(We use $p = 101$, the smallest prime greater than 100.)
For a more detailed discussion of these issues,
see \cite[\S 5]{BremnerPeresi11}.


\section{Normal words in the free dendriform algebra}

An $L$-algebra is a vector space over a field $F$ with two bilinear operations $\prec$, $\succ$
satisfying inner associativity (also called the entanglement identity):
  \begin{equation}
  \label{Lidentity}
  (x \succ y) \prec z \equiv x \succ (y \prec z).
  \end{equation}
Thus a dendriform algebra is an $L$-algebra satisfying two additional identities:
  \begin{align}
  \label{Didentity1}
  (x \prec y) \prec z &\equiv x \prec (y \prec z) + x \prec (y \succ z),
  \\
  \label{Didentity2}
  x \succ (y \succ z) &\equiv (x \succ y) \succ z + (x \prec y) \succ z.
  \end{align}
The free dendriform algebra $\den(X)$ on a set $X$ of generators is the quotient of the free $L$-algebra $L(X)$
by the $T$-ideal generated by these two identities.

We recall some definitions from Bokut et al.~\cite{BokutChenHuang}.
We consider a field $F$, a set of variables $X$, and a set of multilinear operations
  \[
  \Omega = \bigcup_{n \ge 1} \Omega_n,
  \qquad
  \Omega_n = \{ \delta^{(n)}_i \mid i \in I_n \},
  \]
where $\Omega_n$ is the set of $n$-ary operations.
We define the set of $\Omega$-words over $X$ as
  \begin{align*}
  &
  (X,\Omega) = \bigcup_{m=0}^\infty (X,\Omega)_m,
  \quad
  (X,\Omega)_0 = X,
  \quad
  (X,\Omega)_m = X \cup \Omega((X,\Omega)_{m-1}) \; (m \ge 1),
  \\
  &
  \Omega((X,\Omega)_{m-1})
  =
  \bigcup_{t=1}^\infty
  \{ \delta_i^{(t)}(u_1, \dots , u_t) \mid \delta_i^{(t)} \in \Omega_t, \; u_j \in (X,\Omega)_{m-1} \}.
  \end{align*}
Thus $(X,\Omega)_m$ consists of all monomials involving $m$ or fewer operations.

If $\Omega = \{ \succ, \prec \}$ then an $\Omega$-word will be called an $L$-word.
Chen and Wang \cite{ChenWang10} define an $L$-word $u$ to be a normal $L$-word if $u$ is one of the following:
  \begin{itemize}
	\item $ u=x $, where $ x \in X $.
	\item $ u = v \succ w $, where $v$ and $w$ are normal $L$-words.
	\item $ u = v \prec w $, with $ v \neq v_1 \succ v_2 $, where $ v, v_1,v_2,w $ are normal $L$-words.
  \end{itemize}
The set $N$ of normal $L$-words is a linear basis for the free $L$-algebra $L(X)$.
Furthermore, there exists a monomial order in $L(X)$ which allows us to identify the leading normal $L$-word
of any $L$-polynomial.
Chen and Wang \cite{ChenWang10} prove that
  \begin{align*}
  &
  S = \{ f_1(x,y,z), f_2(x,y,z), f_3(x,y,z,v) \mid x,y,z,v \in N \},
  \\[2pt]
  &
  f_1(x,y,z) = (x \prec y) \prec z - x \prec (y \prec z) - x \prec (y \succ z),
  \\
  &
  f_2(x,y,z) = (x \prec y) \succ z + (x \succ y) \succ z - x \succ (y \succ z),
  \\
  &
  f_3(x,y,z,v) = (( x \succ y ) \succ z ) \succ v - ( x \succ y ) \succ ( z \succ v) + ( x \succ ( y \prec z )) \succ v,
  \end{align*}
is a Gr\"obner-Shirshov basis in $L(X)$ for the $T$-ideal generated by the dendriform identities
\eqref{Didentity1}-\eqref{Didentity2}.

Following Chen and Wang \cite{ChenWang10}, we say that an $L$-word $u$ is a normal $D$-word
if $u$ is one of the following:
  \begin{itemize}
	\item $ u=x $, where $ x \in X $.
	\item $ u = x \prec v $, where $x \in X$ and $v$ is a normal $D$-word.
	\item $ u = x \succ v $, where $x \in X$ and $v$ is a normal $D$-word.
	\item $ u = ( x \succ u_1 ) \succ u_2 $, where $x \in X$ and $u_1,u_2$ are normal $D$-words.
  \end{itemize}
The set of normal $D$-words is a linear basis of the free dendriform algebra \cite[Corollary 3.5]{ChenWang10}.
This allows us to reduce substantially the size of the computations in the free dendriform algebra,
as we can see from Table \ref{combinatorialtable}.
The numbers of binary association types (one nonassociative binary operation),
double binary association types (two nonassociative binary operations),
and normal $D$-types, are given by
  \[
  \frac{1}{n}\binom{2n{-}2}{n{-}1},
  \qquad
  2^{n-1} \frac{1}{n}\binom{2n{-}2}{n{-}1},
  \qquad
  \frac{1}{n+1}\binom{2n}{n}.
  \]
The first formula is well-known, and the second is an immediate consequence;
for the third see Loday \cite[\S 5 and Appendix A]{Loday01}.

\begin{table}[h]
\begin{tabular}{r|rrr}
$n$ &\quad binary types &\quad double binary types &\quad normal $D$-types \\
\midrule
3 &\quad   2 &\quad     8 &\quad    5 \\
4 &\quad   5 &\quad    40 &\quad   14 \\
5 &\quad  14 &\quad   224 &\quad   42 \\
6 &\quad  42 &\quad  1344 &\quad  132 \\
7 &\quad 132 &\quad  8448 &\quad  429 \\
8 &\quad 429 &\quad 54912 &\quad 1430
\end{tabular}
\medskip
\caption{Binary association types and normal dendriform words}
\label{combinatorialtable}
\end{table}

Figure \ref{normalizealgorithm} gives an algorithm to express any $L$-word
as a linear combination of normal $D$-words using the reductions given by
the Gr\"{o}bner-Shirshov basis.

\begin{figure}
\begin{itemize}
  \item[] Procedure \textbf{Dnormalize}
  \medskip
  \item[] \textbf{Input:} a term $[c,m]$ consisting of a coefficient $c$ and an $L$-word $m$.
  \item[] \textbf{Output:} a list containing the terms of the normal form of $[c,m]$
  as a linear combination of normal $D$-words.	
  \medskip
	\item If the degree of $m$ is 1 or 2 then
       \begin{itemize}
       \item Set $\texttt{checkedresult} \leftarrow [c,m]$.
       \end{itemize}
	\item Else (\textit{we write $m = m_1 \circ m_2$ where $\circ \in \{ \succ, \prec \}$})
	   \begin{itemize}
	   \item Set $\texttt{result} \leftarrow [\,]$ (\textit{the empty list}).
			\item Set $\texttt{normalized1} \leftarrow \textbf{Dnormalize}([1,m_1])$ (\textit{recursive call}).
			\item Set $\texttt{normalized2} \leftarrow \textbf{Dnormalize}([1,m_2])$ (\textit{recursive call}).
			\item For $[c_1,m'_1]$ in \texttt{normalized1} and $[c_2,m'_2]$ in \texttt{normalized2} do:
				\begin{itemize}
					\item Set $\texttt{newc} \leftarrow c c_1 c_2$.
                    \item Set $\texttt{newm} \leftarrow m'_1 \circ m'_2$.
					\item If \texttt{newm} has the form $( x \succ y ) \prec z$ then:
                          \begin{itemize}
                          \item Append the term $[ \texttt{newc}, x \succ ( y \prec z ) ]$
                          to \texttt{result}
                          \item[] (\textit{reduce using inner associativity}).
                          \end{itemize}
                    \item Else if \texttt{newm} has the form $(x \prec y) \prec z$ then:
                          \begin{itemize}
                          \item Append
                          $[ \texttt{newc}, x \prec (y \prec z) ]$,
                          $[ \texttt{newc}, x \prec (y \succ z) ]$
                          to \texttt{result}
                          (\textit{reduce using $f_1$}).
						\end{itemize}
                    \item Else if \texttt{newm} has the form $(x \prec y) \succ z$ then:
                          \begin{itemize}
                          \item Append
                          $[ -\texttt{newc}, (x {\succ} y) \succ z ]$,
                          $[ \texttt{newc}, x \succ (y {\succ} z) ]$
                          to \texttt{result}
                          (\textit{reduce using $f_2$}).
						\end{itemize}
                    \item Else if \texttt{newm} has the form $(( x \succ y ) \succ z ) \succ v$ then:
                          \begin{itemize}
                          \item Append
                          $[ \texttt{newc}, ( x {\succ} y ) \succ ( z {\succ} v) ]$,
                          $[ -\texttt{newc}, ( x \succ ( y {\prec} z )) \succ v ]$
                          to \texttt{result}
                          (\textit{reduce using $f_3$}).
						\end{itemize}
					\item Else append $[ \texttt{newc}, \texttt{newm} ]$ to \texttt{result}.
				\end{itemize}
            \item If $\texttt{result} = [ [ c, m ] ]$ then (\textit{no reduction was possible})
                  \begin{itemize}
                  \item Set $\texttt{checkedresult} \leftarrow \texttt{result}$.
                  \end{itemize}
            \item Else
                  \begin{itemize}
			      \item Set $\texttt{checkedresult} \leftarrow [\,]$.
			      \item While $\texttt{result} \ne [\,]$ do:
				     \begin{itemize}
					 \item Set $\texttt{element} \leftarrow \texttt{result}[1]$
                           (\textit{the first term in} \texttt{result}).
					 \item Set $\texttt{normalizedelement} \leftarrow \textbf{Dnormalize}( \texttt{element} )$.
					 \item If $\texttt{element} = \texttt{normalizedelement}[1]$ then:
						\begin{enumerate}
							\item[] Append \texttt{element} to \texttt{checkedresult}.
							\item[] Delete \texttt{element} from \texttt{result}.
						\end{enumerate}
					 \item Else in \texttt{result} replace \texttt{element} by \texttt{normalizedelement}.
				     \end{itemize}
                  \end{itemize}
		\end{itemize}
	\item Return \texttt{checkedresult}.
\end{itemize}
\caption{Algorithm to compute the normal form of an $L$-word}
\label{normalizealgorithm}
\end{figure}


\section{Non-existence of identities in degree 3}

Since the defining identities for dendriform algebras have degree 3,
we need to check whether the pre-Jordan product in the free dendriform algebra satisfies
any multilinear identities in degree 3.  In case of a positive answer,
it would be natural to include these identities in the definition of pre-Jordan
algebra.

Since the pre-Jordan product satisfies no identities in degree 2, we need to consider
both possible association types in degree 3 for a binary nonassociative product:
$ ( - \cdot - ) \cdot - $ and $  -\cdot ( - \cdot - ) $.
Applying 6 permutations of 3 variables $x, y, z$ to the 2 association types
gives 12 multilinear nonassociative monomials:
\begin{align*}
	(x \cdot y) \cdot z, \;\; (x \cdot z) \cdot y, \;\; (y \cdot x) \cdot z, \;\;
	(y \cdot z) \cdot x, \;\; (z \cdot x) \cdot y, \;\; (z \cdot y) \cdot x, \\
  	x \cdot (y \cdot z), \;\; x \cdot (z \cdot y), \;\; y \cdot (x \cdot z), \;\;
	y \cdot (z \cdot x), \;\; z \cdot (x \cdot y), \;\; z \cdot (y \cdot x).
\end{align*}
These 12 monomials form a basis of the space $A_3$ of multilinear nonassociative polynomials of degree 3.

In degree 3 there are 8 association types for an algebra with 2 binary nonassociative products:
\begin{align*}
	( - \prec - ) \prec - , \;\;  ( - \prec - ) \succ - , \;\;  ( - \succ - ) \prec - , \;\;  ( - \succ - ) \succ -,
  \\
	 - \prec ( - \prec - ), \;\;   - \succ ( - \prec - ), \;\;   - \prec ( - \succ - ), \;\;   - \succ ( - \succ - ).
\end{align*}
The Gr\"{o}bner-Shirshov basis for the free dendriform algebra \cite{ChenWang10}
allows us to reduce this set to the 5 types, which we call normal $D$-types:
\[
	- \prec ( - \prec - ) , \;\;   - \succ ( - \prec - ), \;\; - \prec ( - \succ - ), \;\; - \succ ( - \succ - ), \;\;  ( - \succ - ) \succ -.
\]
In degree 3 the difference between using all types and using normal $D$-types is not substantial,
but as the degree becomes larger, the difference becomes bigger, and reductions must be done whenever
possible to save computer memory.
Applying 6 permutations of 3 variables to 5 normal $D$-types gives
30 multilinear normal $D$-words, which form a basis of the space $B_3$ of
multilinear dendriform monomials of degree 3.

\begin{proposition}
The pre-Jordan product in the free dendriform algebra satisfies no multilinear identities in degree 3.
\end{proposition}

\begin{proof}
The identities in degree 3 are the nonzero elements of the kernel of the expansion map
$\mathcal{E}_3\colon A_3 \to B_3$.
The $30 \times 12$ expansion matrix $E_3$ represents this linear map with respect to
the monomial bases of $A_3$ and $B_3$ described above.
The $(i,j)$ entry of $E_3$ contains the coefficient of the $i$-th normal $D$-word in the
normalized expansion of the $j$-th pre-Jordan monomial; to normalize the terms of the
expansion, we apply the algorithm in Figure \ref{normalizealgorithm}.

We consider the basic monomial in each association type, by which we mean
the monomial with the identity permutation of the variables.
For both association types, the expansion has 4 terms, but after normalization we
obtain 5 terms:
  \begin{align*}
  (x \cdot y) \cdot z
  &=
  ( x \succ y ) \succ z
  +
  z \prec ( x \succ y )
  +
  ( y \prec x ) \succ z
  +
  z \prec ( y \prec x )
  \\
  &=
  ( x \succ y ) \succ z +
  z \prec ( x \succ y ) -
  ( y \succ x ) \succ z +
  y \succ ( x \succ z ) +
  z \prec ( y \prec x ),
  \\[2pt]
  x \cdot ( y \cdot z )
  &=
  x \succ (y \succ z)
  +
  (y \succ z) \prec x
  +
  x \succ (z \prec y)
  +
  (z \prec y) \prec x
  \\
  &=
  x \succ (y \succ z)
  +
  (y \succ z) \prec x
  +
  x \succ (z \prec y)
  +
  z \prec (y \prec x)
  +
  z \prec (y \succ x) .
  \end{align*}
The expansions of the other basis monomials for $A_3$ are obtained by permutation of the variables.

The resulting expansion matrix is displayed in Table \ref{degree3expmat}.
This matrix has full rank, and hence its nullspace is $\{0\}$.
\end{proof}

\begin{table}
\small
\[
\left[ \begin {array}{rrrrrrrrrrrr}
  . &  . &  . &  . &  . &  1 &  . &  . &  . &  . &  . &  1 \\[-2pt]
  . &  . &  . &  1 &  . &  . &  . &  . &  . &  1 &  . &  . \\[-2pt]
  . &  . &  . &  . &  1 &  . &  . &  . &  . &  . &  1 &  . \\[-2pt]
  . &  1 &  . &  . &  . &  . &  . &  1 &  . &  . &  . &  . \\[-2pt]
  . &  . &  1 &  . &  . &  . &  . &  . &  1 &  . &  . &  . \\[-2pt]
  1 &  . &  . &  . &  . &  . &  1 &  . &  . &  . &  . &  . \\[-2pt]
  . &  . &  . &  . &  . &  . &  . &  1 &  . &  . &  1 &  . \\[-2pt]
  . &  . &  . &  . &  . &  . &  1 &  . &  1 &  . &  . &  . \\[-2pt]
  . &  . &  . &  . &  . &  . &  . &  . &  . &  1 &  . &  1 \\[-2pt]
  . &  . &  . &  . &  . &  . &  1 &  . &  1 &  . &  . &  . \\[-2pt]
  . &  . &  . &  . &  . &  . &  . &  . &  . &  1 &  . &  1 \\[-2pt]
  . &  . &  . &  . &  . &  . &  . &  1 &  . &  . &  1 &  . \\[-2pt]
  . &  . &  . &  1 &  . &  . &  . &  . &  . &  . &  . &  1 \\[-2pt]
  . &  . &  . &  . &  . &  1 &  . &  . &  . &  1 &  . &  . \\[-2pt]
  . &  1 &  . &  . &  . &  . &  . &  . &  . &  . &  1 &  . \\[-2pt]
  . &  . &  . &  . &  1 &  . &  . &  1 &  . &  . &  . &  . \\[-2pt]
  1 &  . &  . &  . &  . &  . &  . &  . &  1 &  . &  . &  . \\[-2pt]
  . &  . &  1 &  . &  . &  . &  1 &  . &  . &  . &  . &  . \\[-2pt]
  . &  . &  1 &  . &  . &  . &  1 &  . &  . &  . &  . &  . \\[-2pt]
  . &  . &  . &  . &  1 &  . &  . &  1 &  . &  . &  . &  . \\[-2pt]
  1 &  . &  . &  . &  . &  . &  . &  . &  1 &  . &  . &  . \\[-2pt]
  . &  . &  . &  . &  . &  1 &  . &  . &  . &  1 &  . &  . \\[-2pt]
  . &  1 &  . &  . &  . &  . &  . &  . &  . &  . &  1 &  . \\[-2pt]
  . &  . &  . &  1 &  . &  . &  . &  . &  . &  . &  . &  1 \\[-2pt]
  1 &  . & -1 &  . &  . &  . &  . &  . &  . &  . &  . &  . \\[-2pt]
  . &  1 &  . &  . & -1 &  . &  . &  . &  . &  . &  . &  . \\[-2pt]
 -1 &  . &  1 &  . &  . &  . &  . &  . &  . &  . &  . &  . \\[-2pt]
  . &  . &  . &  1 &  . & -1 &  . &  . &  . &  . &  . &  . \\[-2pt]
  . & -1 &  . &  . &  1 &  . &  . &  . &  . &  . &  . &  . \\[-2pt]
  . &  . &  . & -1 &  . &  1 &  . &  . &  . &  . &  . &  .
\end {array} \right]
\]
\caption{Expansion matrix in degree 3}
\label{degree3expmat}
\end{table}

\begin{remark} \label{remarkdegree3}
Throughout this paper we will be comparing the identities satisfied by the pre-Jordan product
in the free dendriform algebra with the identities satisfied by the Jordan diproduct in the
free associative dialgebra \cite{BremnerPeresi11}.
Every multilinear identity of degree 3
satisfied by the Jordan diproduct in the free associative dialgebra is a consequence of the
right commutative identity $x(yz) \equiv x(zy)$ \cite[Theorem 6]{Bremner10}.
This explains why in \cite{BremnerPeresi11} the domain of the expansion map in higher degrees
was taken to be the multilinear subspace of the free right commutative algebra.
However, in this paper, since we have no identities in degree 3, the domain of the expansion
map will be the multilinear subspace of the free nonassociative algebra.
\end{remark}


\section{Existence of defining identities in degree 4}

To find the multilinear polynomial identities in degree 4 satisfied by the pre-Jordan product in the
free dendriform algebra, we proceed as in degree 3, but the expansion matrix $E_4$ is larger.
We show that the defining identities for pre-Jordan algebras generate the kernel of the expansion map
$\mathcal{E}_4$ as an $S_4$-module.

In degree 4 there are 5 association types for a nonassociative binary product:
  \[
  ( ( - \cdot - ) \cdot - ) \cdot -,
  \quad
  ( - \cdot ( - \cdot - ) ) \cdot -,
  \quad
  ( - \cdot - ) \cdot ( - \cdot - ),
  \quad
  - \cdot ( ( - \cdot - ) \cdot - ),
  \quad
  - \cdot ( - \cdot ( - \cdot - ) ).
  \]
Applying 24 permutations of 4 variables $x, y, z, v$ we get 120 multilinear monomials
which form a basis of $A_4$.
We order these monomials first by association type and then by lexicographical order of the
permutation.

In an algebra with 2 nonassociative binary operations, there are 40 association types,
which reduce to 14 normal $D$-types:
  \begin{equation}
  \label{normalDtypes}
  \left\{
  \begin{array}{lll}
  - \prec (- \prec (- \prec -)), &\;\;
  - \succ (- \prec (- \prec -)), &\;\;
  - \prec (- \succ (- \prec -)), \\[2pt]
  - \succ (- \succ (- \prec -)), &\;\;
  - \prec (- \prec (- \succ -)), &\;\;
  - \succ (- \prec (- \succ -)), \\[2pt]
  - \prec (- \succ (- \succ -)), &\;\;
  - \succ (- \succ (- \succ -)), &\;\;
  - \prec ((- \succ -) \succ -), \\[2pt]
  - \succ ((- \succ -) \succ -), &\;\;
  (- \succ -) \succ (- \prec -), &\;\;
  (- \succ -) \succ (- \succ -), \\[2pt]
  (- \succ (- \prec -)) \succ -, &\;\;
  (- \succ (- \succ -)) \succ -.
  \end{array}
  \right.
  \end{equation}
Applying permutations of the variables we get 336 multilinear normal $D$-words
which form a basis for the space $B_4$.
We order these normal $D$-words first by $D$-type and then by lex order of the
permutation.

\begin{theorem} \label{pj4}
The kernel of the expansion map $\mathcal{E}_4\colon A_4 \to B_4$ is generated as
an $S_4$-module by the defining identities for pre-Jordan algebras:
  \begin{align*}
  PJ_1(x,y,z,u) &= (x \circ y) \cdot (z \cdot u) + (y \circ z) \cdot (x \cdot u) + (z \circ x) \cdot (y \cdot u)
  \\
  &\quad
  - z \cdot [(x \circ y) \cdot u] - x \cdot [(y \circ z) \cdot u] - y \cdot [(z \circ x) \cdot u]
  \equiv 0,
  \\[2pt]
  PJ_2(x,y,z,u) &= x \cdot [y \cdot (z \cdot u)] + z \cdot [y \cdot (x \cdot u)] + [(x \circ z) \circ y] \cdot u
  \\
  &\quad
  - z \cdot [(x \circ y) \cdot u] - x \cdot [(y \circ z) \cdot u] - y \cdot [(z \circ x) \cdot u]
  \equiv 0,
  \end{align*}
where $ x \circ y = x \cdot y + y \cdot x $.
\end{theorem}

\begin{proof}
The $336 \times 120$ matrix $E_4$ represents the expansion map $\mathcal{E}_4\colon A_4 \to B_4$
with respect to the given ordered bases of $A_4$ and $B_4$.
The $(i,j)$ entry of $E_4$ is the coefficient of the $i$-th normal $D$-word in the normalized expansion
of the $j$-th nonassociative monomial.
We display the expansions of the 5 association types with the identity permutation of
the variables.
Each expansion produces 8 terms, but normalization increases the number to 15, 13, 12, 14, 16 terms
respectively:
\begin{align*}
((x \cdot y) \cdot z) \cdot u
&=
((x \succ y) \succ z) \succ u
+
u \prec ((x \succ y) \succ z)
+
(z \prec (x \succ y)) \succ u \\[-2pt]
&\quad
+
u \prec (z \prec (x \succ y))
+
((y \prec x) \succ z) \succ u
+
u \prec ((y \prec x) \succ z) \\[-2pt]
&\quad +
(z \prec (y \prec x)) \succ u
+
u \prec (z \prec (y \prec x)) \\
& =
(x \succ y) \succ (z \succ u)
-
(x \succ (y \prec z)) \succ u
+
u \prec ((x \succ y) \succ z) \\[-2pt]
&\quad -
(z \succ (x \succ y)) \succ u
+
z \succ ((x \succ y) \succ u)
+
u \prec (z \prec (x \succ y)) \\[-2pt]
&\quad -
(y \succ x) \succ (z \succ u)
+
(y \succ (x \prec z)) \succ u
+
(y \succ (x \succ z)) \succ u \\[-2pt]
&\quad -
u \prec ((y \succ x) \succ z)
+
u \prec (y \succ (x \succ z))
-
(z \succ (y \prec x)) \succ u  \\[-2pt]
&\quad -
z \succ ((y \succ x) \succ u)
+
z \succ (y \succ (x \succ u))
+
u \prec (z \prec (y \prec x)), \\[2pt]
(x \cdot (y \cdot z)) \cdot u
& =
(x \succ (y \succ z)) \succ u
+
u \prec (x \succ (y \succ z))
+
((y \succ z) \prec x) \succ u \\[-2pt]
&\quad +
u \prec ((y \succ z) \prec x)
+
(x \succ (z \prec y)) \succ u
+
u \prec (x \succ (z \prec y)) \\[-2pt]
&\quad +
((z \prec y) \prec x) \succ u
+
u \prec ((z \prec y) \prec x) \\
& =
(x \succ (y \succ z)) \succ u
+
u \prec (x \succ (y \succ z))
+
(y \succ (z \prec x)) \succ u \\[-2pt]
&\quad +
u \prec (y \succ (z \prec x))
+
(x \succ (z \prec y)) \succ u
+
u \prec (x \succ (z \prec y)) \\[-2pt]
&\quad -
(z \succ (y \prec x)) \succ u
-
z \succ ((y \succ x) \succ u)
+
z \succ (y \succ (x \succ u)) \\[-2pt]
&\quad -
(z \succ (y \succ x)) \succ u
+
z \succ ((y \succ x) \succ u)
+
u \prec (z \prec (y \prec x)) \\[-2pt]
&\quad +
u \prec (z \prec (y \succ x)),
\\[2pt]
(x \cdot y) \cdot (z \cdot u)
& =
(x \succ y) \succ (z \succ u)
+
(z \succ u) \prec (x \succ y)
+
(x \succ y) \succ (u \prec z) \\[-2pt]
&\quad +
(u \prec z) \prec (x \succ y)
+
(y \prec x) \succ (z \succ u)
+
(z \succ u) \prec (y \prec x) \\[-2pt]
&\quad +
(y \prec x) \succ (u \prec z)
+
(u \prec z) \prec (y \prec x) \\
& =
(x \succ y) \succ (z \succ u)
+
z \succ (u \prec (x \succ y))
+
(x \succ y) \succ (u \prec z) \\[-2pt]
&\quad +
u \prec (z \prec (x \succ y))
+
u \prec (z \succ (x \succ y))
-
(y \succ x) \succ (z \succ u) \\[-2pt]
&\quad +
y \succ (x \succ (z \succ u))
+
z \succ (u \prec (y \prec x))
-
(y \succ x) \succ (u \prec z) \\[-2pt]
&\quad +
y \succ (x \succ (u \prec z))
+
u \prec (z \prec (y \prec x))
+
u \prec (z \succ (y \prec x)),
\\[2pt]
x \cdot ((y \cdot z) \cdot u)
& =
x \succ ((y \succ z) \succ u)
+
((y \succ z) \succ u) \prec x
+
x \succ (u \prec (y \succ z)) \\[-2pt]
&\quad +
(u \prec (y \succ z)) \prec x
+
x \succ ((z \prec y) \succ u)
+
((z \prec y) \succ u) \prec x \\[-2pt]
&\quad +
x \succ (u \prec (z \prec y))
+
(u \prec (z \prec y)) \prec x \\
& =
x \succ ((y \succ z) \succ u)
+
(y \succ z) \succ (u \prec x)
+
x \succ (u \prec (y \succ z)) \\[-2pt]
&\quad +
u \prec (y \succ (z \prec x))
+
u \prec ((y \succ z) \succ x)
-
x \succ ((z \succ y) \succ u) \\[-2pt]
&\quad +
x \succ (z \succ (y \succ u))
-
(z \succ y) \succ (u \prec x)
+
z \succ (y \succ (u \prec x)) \\[-2pt]
&\quad +
x \succ (u \prec (z \prec y))
+
u \prec (z \prec (y \prec x))
+
u \prec (z \prec (y \succ x)) \\[-2pt]
&\quad -
u \prec ((z \succ y) \succ x)
+
u \prec (z \succ (y \succ x)),
\\[2pt]
x \cdot (y \cdot (z \cdot u))
& =
x \succ (y \succ (z \succ u))
+
(y \succ (z \succ u)) \prec x
+
x \succ ((z \succ u) \prec y) \\[-2pt]
&\quad +
((z \succ u) \prec y) \prec x
+
x \succ (y \succ (u \prec z))
+
(y \succ (u \prec z)) \prec x \\[-2pt]
&\quad +
x \succ ((u \prec z) \prec y)
+
((u \prec z) \prec y) \prec x \\
& =
x \succ (y \succ (z \succ u))
+
y \succ (z \succ (u \prec x))
+
x \succ (z \succ (u \prec y)) \\[-2pt]
&\quad +
z \succ (u \prec (y \prec x))
+
z \succ (u \prec (y \succ x))
+
x \succ (y \succ (u \prec z)) \\[-2pt]
&\quad +
y \succ (u \prec (z \prec x))
+
y \succ (u \prec (z \succ x))
+
x \succ (u \prec (z \prec y)) \\[-2pt]
&\quad +
x \succ (u \prec (z \succ y))
+
u \prec (z \prec (y \prec x))
+
u \prec (z \prec (y \succ x)) \\[-2pt]
&\quad -
u \prec ((z \succ y) \succ x)
+
u \prec (z \succ (y \succ x))
+
u \prec (z \succ (y \prec x)) \\[-2pt]
&\quad +
u \prec ((z \succ y) \succ x).
\end{align*}
We compute the rank of the expansion matrix and obtain 104, so the nullity is 16,
showing that there are identities in degree 4.

The most natural way to obtain an integer basis for the nullspace of an integer matrix
is to use the Hermite normal form.
In the present case, this produces an invertible $120 \times 120$ integer matrix $U$
for which $U E_4^t = H$ where $H$ is the Hermite normal form of $E_4^t$.
Since $E_4$ has rank 104, the last 16 rows of $H$ are zero, and so the last 16 rows of $U$
form a basis for the left nullspace of $E_4^t$ which is the right nullspace of $E_4$.
The squared Euclidean lengths of these nullspace basis vectors are 12 (4 times), 24 (6 times),
36 (3 times), 48 (3 times).

To get shorter vectors, we apply the LLL algorithm for lattice basis reduction \cite{BremnerBook}
to these nullspace basis vectors.
This gives 16 vectors all with squared Euclidean length 12,
which we denote by $V_1, \dots, V_{16}$.
For further details on this method as applied to polynomial identities, see
\cite{BremnerPeresi09b}.

We compare the resulting identities (whose coefficient vectors are the nullspace basis vectors)
with the defining identities for pre-Jordan algebras to verify that they are equivalent;
that is, that they generate the same $S_4$-module.
We construct a matrix with 144 rows and 120 columns,
consisting of a $120 \times 120$ upper block and a $24 \times 120$ lower block.
For each of the 2 defining identities for pre-Jordan algebras,
we fill the lower block with the coefficient vectors obtained by applying all 24 permutations of the variables
and then compute the row canonical form (RCF).
After processing both identities, the rank is 16.
Retaining the results of this computation, we perform the same calculation
for each of the 16 nullspace basis identities described above.
The rank does not increase, and hence the nullspace basis identities are consequences of the
defining identities for pre-Jordan algebras.

We then reverse this calculation, first processing the nullspace basis identities, obtaining rank 16,
and then processing the defining identities, which do not increase the rank.
Hence the defining identities are consequences of the nullspace basis identities.
This completes the proof.
\end{proof}

\subsection*{$S_4$-module structure of identities}

As mentioned in Remark \ref{remarkdegree3}, we will compare the identities for the pre-Jordan product
in the free dendriform algebra with those for the Jordan diproduct in the free associative dialgebra
\cite{BremnerPeresi11}.
In degree 4, both spaces are nonzero, so we can explicitly compute the structure of the identities as
modules over the symmetric group $S_4$.
Table \ref{S4table} gives the character table of $S_4$.

\begin{table}[h]
\[
\begin{array}{l|ccccc}
	\text{partition} & (1)(2)(3)(4) & (12)(3)(4) & (12)(34) & (123)(4) & (1234) \\
	\hline
	4 & 1 & 1 & 1 & 1 & 1 \\
	31 & 3 & 1 & -1 & 0 & -1 \\
	22 & 2 & 0 & 2 & -1 & 0\\
	211 & 3 & -1 & -1 & 0 & 1 \\
	1111 & 1 & -1 & 1 & 1 & -1 \\[2pt]
\end{array}
\]
\medskip
\caption{Character table of the symmetric group $S_4$}
\label{S4table}
\end{table}

We write $N_4$ for the nullspace of the expansion matrix $E_4$.
From the last proof we know that the reduced vectors $V_1, \dots, V_{16}$ form a basis for $N_4$.
Each of these vectors has squared length 12, and has 12 nonzero components from $\{ \pm 1 \}$.

To compute the character of $N_4$ as an $S_4$-module, we choose a set of conjugacy class representatives
(the column labels in Table \ref{S4table}), and compute the matrix representing the action of
each representative on $N_4$ with respect to the basis $V_1, \dots, V_{16}$.
The traces of these matrices give the character $( 16,4,0,1,0 )$ of the $S_4$-module $N_4$.
Expressing this character as a linear combination of the rows of the character table
gives the decomposition of $N_4$ as a sum of irreducible $S_4$-submodules:
  \begin{equation}
  \label{N4decomp}
  N_4 \cong 2[4] \oplus 3[31] \oplus [22] \oplus [211],
  \end{equation}
where $m [\lambda]$ denotes $m$ copies of the irreducible module for partition $\lambda$.
We note that there are no copies of the signature module $[1111]$.

For the Jordan diproduct in the free associative dialgebra, the domain $A_4$ of the expansion map
$\mathcal{E}_4$ is the multilinear subspace in degree 4 of the free right commutative algebra,
and $B_4$ is the multilinear subspace in degree 4 of the free associative dialgebra.
A basis for the kernel of this expansion map is given by
the 16 rows of the lower block in the matrix of \cite[Table 4]{Bremner10}.
Following the methods of the previous paragraph, we determine that the nullspace for the Jordan
diproduct has the same character as the nullspace for the pre-Jordan product.

\begin{proposition}
In degree 4, the $S_4$-module of multilinear nonassociative polynomial identities satisfied by
the pre-Jordan product in the free dendriform algebra is isomorphic to the $S_4$-module of
multilinear right-commutative polynomial identities satisfied by the Jordan diproduct in the
free associative dialgebra.
\end{proposition}

As we will see throughout this paper, this fact is not accidental:
in each degree $4 \le n \le 8$, there is an $S_n$-module isomorphism
between the new identities satisfied by the pre-Jordan product in the free dendriform algebra
and the new identities satisfied by the quasi-Jordan product in the free associative dialgebra.
(An identity is called new if it does not follow from identities of lower degree.)


\section{Representation theory of the symmetric group}

Since the sizes of the matrices we use are growing, it is useful to introduce another method
to compute and compare identities.
This new method is more sophisticated and involves deeper knowledge of the representation theory
of the symmetric group.
The basic idea consists in breaking down a module of identities into its irreducible components
and doing the computations in these smaller submodules.
Furthermore, using representation theory allows us to consider only the basic monomial in each
association type (the monomial with the identity permutation of the variables);
without representation theory, we have to use all permutations of the variables, which significantly
increases the size of the computations.

\subsection*{Theoretical background}

Given $f(x_1, \dots, x_n)$, a multilinear nonassociative polynomial of degree $n$,
which we regard as a polynomial identity for an algebra,
we collect its terms by association type and write $f = f_1 + \dots + f_t$,
where $t$ is the number of association types.
In each association type a term can be identified by its coefficient $c \in F$ and
the permutation $\pi$ of the variables in the monomial.
Thus each $f_i$ can be expressed as an element
  \[
  g_i = \sum_{\pi \in S_n} c_{i \pi} \pi,
  \]
of the group algebra $FS_n$, and hence $f$ can be identified with the element
$(g_1, \dots, g_t)$ of $M = FS_n \oplus \cdots \oplus FS_n$ ($t$ summands).
If $\pi \in S_n$ then $(\pi g_1, \dots, \pi g_t)$ is also an identity,
which represents the identity $f$ applied to a permutation of its arguments.
Since linear combinations of identities are also identities, it follows that
$(gg_1, \dots, gg_t)$ is an identity for any element $g \in FS_n$.
Thus $M$ is a module for the group algebra $FS_n$ and the set of identities
which are consequences of $f$ is a submodule of $M$.

The partitions $\lambda$ of $n$ are in bijection with the isomorphism classes
of irreducible representations of $S_n$;
we write $d_\lambda$ for the dimension of the irreducible representation corresponding to $\lambda$.
By Wedderburn's theorem, in the case $F = \mathbb{Q}$,
the group algebra $\mathbb{Q}S_n$ is isomorphic to the direct sum of matrix algebras
of size $d_\lambda \times d_\lambda$:
\begin{align*}
	\rho\colon
\mathbb{Q} S_n \overset{\cong}{\longrightarrow}
	\underset{\lambda}{\bigoplus}
	M_{d_\lambda}(\mathbb{Q}).
\end{align*}
If $\rho_\lambda$ is the projection of the group algebra onto the matrix subalgebra corresponding to
the partition $\lambda$, then we associate to each element of the group algebra a matrix of size
$d_\lambda \times d_\lambda$.
Thus $[ \, \rho_\lambda(g_1) \, \cdots \, \rho_\lambda(g_t) \, ]$ is a matrix of size
$d_\lambda \times td_\lambda$ which corresponds to the identity $f$ in the representation associated
to $\lambda$.
In the same way,
  \[
  [ \, \rho_\lambda(gg_1) \, \cdots \, \rho_\lambda(gg_t) \, ]
  =
  \rho_\lambda(g)
  [ \, \rho_\lambda(g_1) \, \cdots \, \rho_\lambda(g_t) \, ]
  \]
corresponds to a sequence of row operations applied to the matrix representing $f$.
(We are primarily interested in computing the RCFs of these matrices,
which includes the possibility of a nonzero row being reduced to zero.)
The matrices $\rho_\lambda(g)$ are computed using the algorithm of Clifton \cite{Clifton};
see also \cite[Figure 1]{BremnerPeresi11}.

Two identities are equivalent if and only if they generate the same submodule of $M$.
In other words, for each irreducible representation $\lambda$ the matrices representing the two
identities have the same row space.
Stacking vertically the matrices corresponding to $k$ identities $f^{(1)}, \dots, f^{(k)}$ gives a
matrix of size $k d_\lambda \times t d_\lambda$. Each row of this matrix represents an identity implied
by $f^{(1)}, \dots, f^{(k)}$. Row operations on this matrix replace rows with linear combination of rows.
Thus the nonzero rows of the RCF are a set of linearly independent
generators for the submodule of $M$ generated by $f^{(1)}, \dots, f^{(k)}$ corresponding to
partition $\lambda$.
Each of these nonzero rows represents a copy of the irreducible module $[\lambda]$
in the decomposition of this submodule.

\subsection*{Comparison with other methods}

The method used in \cite{BremnerPeresi11} (see Table 3) could be used to find 
all the polynomial identities in degree $n$ satisfied by the pre-Jordan product
in the free dendriform algebra, as follows.
Let $t$ be the number of binary nonassociative types in degree $n$,
and let $s$ be the number of normal $D$-types in degree $n$.
Let $\lambda$ be a partition of $n$ with irreducible representation of dimension $d_\lambda$.
We write $X_\lambda$ for the matrix of size $t d_\lambda \times s d_\lambda$ consisting of
$d_\lambda \times d_\lambda$ blocks.
To fill this matrix, we first compute the normalized expansions of the basic nonassociative
monomials in the $t$ association types.
For each $i = 1, \dots, t$ and $j = 1, \dots, s$ we collect the terms in the expansion of the
$i$-th basic monomial which have the $j$-th normal $D$-type.
In the matrix $X_\lambda$, the block in position $(i,j)$ is the representation matrix for
the group algebra element corresponding to the $j$-th normal $D$-type in the $i$-th expansion.
We now form the augmented matrix $[X_\lambda|I]$ where $I$ is the identity matrix of order $t d_\lambda$.
Assume that $X_\lambda$ has rank $r$.
We compute the RCF of  $[X_\lambda|I]$ to obtain $[R|U]$
where $R$ is the RCF of $X_\lambda$ and $U$ is an invertible matrix for which $U X_\lambda = R$.
The last $td_\lambda-r$ rows of $U$ represent the identities in degree $n$ and partition $\lambda$
satisfied by the pre-Jordan product in the free dendriform algebra.
Then we also have $X_\lambda = U^{-1} R$.
Since the last $td_\lambda-r$ rows of $R$ are zero, the first $r$ columns of $U^{-1}$
show how to express the rows of $X_\lambda$ in terms of the nonzero rows of $R$,
and since the nonzero rows of $R$ are linearly independent,
the first $r$ columns of $U^{-1}$ are uniquely determined.
Hence the first $r$ rows of $U$ are uniquely determined.
Moreover, the last $td_\lambda-r$ rows of U are uniquely determined since they form a submatrix in RCF
(these rows are the lower right block which has the identities we are looking for).
Since $UX_\lambda = R$, it follows that rows $r+1$ to $td_\lambda$ of $U$ form a (canonical) basis of
the left nullspace of $X_\lambda$, which coincides with the right nullspace of $X_\lambda^t$.
So if we compute the RCF of $X_\lambda^t$, find the canonical basis of its nullspace, put these basis
vectors into another matrix, and then compute its RCF, we will obtain the same matrix as the
last $td_\lambda-r$ rows of $U$.
We can therefore obtain the same results using $X_\lambda^t$,
which in general is much smaller than $[X_\lambda|I]$.

\subsection*{Example: degree 4}

We consider again the identities in degree 4 for the pre-Jordan product
in the free dendriform algebra.

There are 5 partitions of 4, namely 4, 31, 22, 211, 1111; the corresponding irreducible representations
of $S_4$ have dimensions 1, 3, 2, 3, 1 respectively.
For each partition $\lambda$ with corresponding dimension $d_\lambda$, we construct a matrix $X_4$
of size $5 d_\lambda \times 14 d_\lambda$.
This matrix consists of 5 rows and 14 columns of $d_\lambda \times d_\lambda$ blocks.
To fill this matrix, we first compute the normalized expansions of the basic nonassociative monomials
in the 5 association types.
For each $i = 1, \dots, 5$ and $j = 1, \dots, 14$ we collect the terms in the expansion of the
$i$-th basic monomial which have the $j$-th normal $D$-type.
For example, for $i = 1$ the sorted expansion of $((x \cdot y) \cdot z) \cdot u$ is
  \begin{align*}
  &
  \big[
  u  z  y  x
  \big]_1
  +
  \big[
  u  z  x  y
  \big]_5
  +
  \big[
  u  y  x  z
  \big]_7
  +
  \big[
  z  y  x  u
  \big]_8
  +
  \big[
  u  x  y  z
  -
  u  y  x  z
  \big]_9
  +
  \big[
  z  x  y  u
  -
  z  y  x  u
  \big]_{10}
  \\
  &
  +
  \big[
  x  y  z  u
  -
  y  x  z  u
  \big]_{12}
  +
  \big[
  -
  x  y  z  u
  +
  y  x  z  u
  -
  z  y  x  u
  \big]_{13}
  +
  \big[
  -
  z  x  y  u
  +
  y  x  z  u
  \big]_{14}.
  \end{align*}
The subscripts indicate the normal $D$-type in the order given in \eqref{normalDtypes};
within each pair of brackets, we give the corresponding element of the group algebra $\mathbb{Q} S_4$.
In the matrix $X_4$, the block in position $(i,j)$ is the representation matrix for
the group algebra element corresponding to the $j$-th normal $D$-type in the $i$-th expansion.
The nullspace of the transpose matrix $X_4^t$ represents the identities for the pre-Jordan
corresponding to representation $[\lambda]$.

In particular, we consider $\lambda = 4$ with $d_\lambda = 1$.
Isomorphism \eqref{N4decomp} shows that the representation $[4]$ has multiplicity 2
in the $S_4$-module of identities in degree 4 for the pre-Jordan product.
Thus we expect that the nullspace of $X_4^t$ will be 2-dimensional.
We calculate $X_4$,
  \[
  \left[
  \begin {array}{rrrrrrrrrrrrrr}
  1&0&0&0&1&0&1&1&0&0&0&0&-1&0 \\
  1&0&2&0&1&0&1&1&0&0&0&0& 1&0 \\
  1&1&1&1&1&1&1&1&0&0&0&0& 0&0 \\
  1&1&1&1&1&1&1&1&0&0&0&0& 0&0 \\
  1&3&1&3&1&3&1&1&0&0&0&0& 0&0
  \end {array}
  \right],
  \]
and obtain the row canonical form of $X_4^t$ (omitting zero rows),
  \[
  \left[ \begin {array}{rrrrr}
  1&0&0&0&-1 \\
  0&1&0&0&-1 \\
  0&0&1&1& 3
  \end {array}
  \right],
  \]
showing that the nullity is 2.

Similar calculations with $\lambda = 31, 22, 211, 1111$ produce nullities 3, 1, 1, 0 respectively,
agreeing with the isomorphism \eqref{N4decomp}.


\section{Non-existence of special identities in degrees 5, 6 and 7}

The pre-Jordan product satisfies the defining identities for pre-Jordan algebras in degree 4.
It follows that the pre-Jordan product satisfies identities in every degree $n \ge 4$,
which we call the liftings of the defining identities.

\begin{definition}
From a multilinear polynomial identity $f( x_1, x_2, \dots, x_n )$ in degree $n$
for a nonassociative algebra with one binary operation $x \cdot y$,
we obtain $n+2$ identities in degree $n+1$ which generate all the consequences
of $f$ in degree $n+1$ as a module over the symmetric group $S_{n+1}$.
We introduce a new variable $x_{n+1}$ and perform $n$ substitutions and 2 multiplications,
obtaining the \textbf{liftings} of $f$ to degree $n+1$:
  \begin{align*}
	&
	f( x_1 \cdot x_{n+1}, x_2, \dots, x_n ),
	\;\;
	f( x_1, x_2 \cdot x_{n+1}, \dots, x_n ),
	\;\;
	\dots\,,
	\;\;
	f( x_1, x_2, \dots, x_n \cdot x_{n+1} ),
	\\
	&
	f( x_1, x_2, \dots, x_n ) \cdot x_{n+1},
	\;\;
	x_{n+1} \cdot f( x_1, x_2, \dots, x_n ).
  \end{align*}
\end{definition}

\begin{definition}
Let $\Omega = \{ \omega_1, \dots, \omega_\ell \}$ be operation symbols of arbitrary arities,
and let $\mathcal{V}$ be the variety of all $\Omega$-algebras.
Let $\mathcal{W}$ be a variety of algebras defined by multilinear identities.
Let $P = \{ p_1, \dots, p_\ell \}$ be multilinear polynomials in $\mathcal{W}$ such that
the degree of $p_i$ equals the arity of $\omega_i$.
The \textbf{expansion map} in degree $n$ has domain equal to the multilinear subspace in degree $n$
of the free $\mathcal{V}$-algebra and codomain equal to the multilinear subspace in degree $n$
of the free $\mathcal{W}$-algebra; it is defined by replacing every occurrence of $\omega_i$ by $p_i$.
We choose a degree $m$ and consider the kernels of the expansion maps in all degrees $n \le m$.
We write $f_1, \dots, f_k$ for a set of $S_n$-module generators of these kernels.
Any element of the kernel of the expansion map in degree $n' > m$ which is not a consequence
of the liftings of $f_1, \dots, f_k$ to degree $n'$ will be called a \textbf{special identity}
for the operations $P$ in every $\mathcal{W}$-algebra.
\end{definition}

In our case, $\Omega = \{ x \cdot y \}$ consists of a single bilinear operation,
$\mathcal{W}$ is the variety of dendriform algebras,
$P = \{ x \succ y + y \prec x \}$ consists of the pre-Jordan product in the free dendriform algebra,
and $m = 4$ (the degree of the defining identities for pre-Jordan algebras).

\subsection*{Degree 5}

We first need to lift the defining identities $PJ_1$, $PJ_2$ for pre-Jordan algebras (Theorem \ref{pj4})
to degree 5.
For each $i = 1, 2$ we perform 4 substitutions and 2 multiplications, obtaining 12 identities in degree 5:
  \begin{equation}
  \label{lift5}
  \left\{
  \begin{array}{lll}
  PJ_i( x \cdot v, y, z, u ), &\quad
  PJ_i( x, y \cdot v, z, u ), &\quad
  PJ_i( x, y, z \cdot v, u ),
  \\[2pt]
  PJ_i( x, y, z, u \cdot v ), &\quad
  PJ_i( x, y, z, u ) \cdot v, &\quad
  v \cdot PJ_i( x, y, z, u ).
  \end{array}
  \right.
  \end{equation}
In degree 5 we have $t = 14$ association types for the binary nonassociative product.
In the free dendriform algebra we have $s = 42$ normal $D$-types.
There are 7 distinct irreducible representations of the symmetric group $S_5$,
with dimensions $d_\lambda = 1$, 4, 5, 6, 5, 4, 1 corresponding to partitions $\lambda = 5$,
41, 32, 311, 221, 2111, 11111.
The representation matrices $\rho_i(\pi)$ ($1 \le i \le 7$)
are given by the projections onto the corresponding simple ideals in the direct sum
decomposition of the group algebra:
  \[
  \mathbb{Q} S_5
  =
  \mathbb{Q} \oplus
  M_4(\mathbb{Q}) \oplus
  M_5(\mathbb{Q}) \oplus
  M_6(\mathbb{Q}) \oplus
  M_5(\mathbb{Q}) \oplus
  M_4(\mathbb{Q}) \oplus
  \mathbb{Q}.
  \]
For each partition $\lambda$ we construct a block matrix $X_\lambda$ as described in the previous section.
It has size $14 d_\lambda \times 42 d_\lambda$;
the rows of blocks are labeled by the basic nonassociative monomials of degree 5,
and the columns of blocks are labeled by the normal $D$-types of degree 5.
Thus each row of blocks contains the representation matrices of the terms of the expansion
of the corresponding basic monomial into the free dendriform algebra.
We compute $\mathrm{RCF}(X_\lambda^t)$,
find the rank and the nullity of $X_\lambda^t$,
extract a basis for the nullspace of $X_\lambda^t$,
put the basis vectors into the rows of another matrix $N_\lambda$,
and compute $\mathrm{RCF}(N_\lambda)$.
The nonzero rows of $\mathrm{RCF}(N_\lambda)$ give a canonical set of generators for the $S_5$-module
of identities corresponding to partition $\lambda$.
See the 4 columns labelled ``all'' in Table \ref{tableranksdegree5}.

We need to check that these identities are consequences of the pre-Jordan identities.
Thus we compare them with the liftings to degree 5 of the defining identities for pre-Jordan algebras.
For each partition $\lambda$ of 5 we construct a $12d_\lambda \times 14d_\lambda$ matrix $L_\lambda$.
The rows of blocks correspond to the 12 liftings \eqref{lift5} of the defining identities
for pre-Jordan algebras.
The columns of blocks are labeled by the normal $D$-types of degree 5.
For each $i = 1, \dots, 12$ and $j = 1, \dots, 14$ we collect the terms in the $i$-th lifting
which have the $j$-th normal $D$-type, and store the corresponding representation matrix in
the block in position $(i,j)$.
We compute $\mathrm{RCF}(L_\lambda)$;
the nonzero rows of $\mathrm{RCF}(L_\lambda)$ give a canonical set of generators for the $S_5$-module
of lifted identities corresponding to partition $\lambda$.
See the 3 columns labelled ``lifted'' in Table~\ref{tableranksdegree5}.

For every partition $\lambda$, the rank of the lifted identities equals the nullity of all identities;
moreover, $\mathrm{RCF}(N_\lambda) = \mathrm{RCF}(L_\lambda)$, omitting zero rows.
We conclude that there are no special identities in degree 5
for the pre-Jordan product in the free dendriform algebra.

\begin{table}
   \begin{center}
   \begin{tabular}{rlr|rrr|rrrr|r}
   & & & \multicolumn{3}{|c|}{lifted ($L_\lambda$)} & \multicolumn{4}{|c|}{all ($X_\lambda^t$)} \\
   \# & $\lambda$ & $d_\lambda$ & rows & cols & rank &  rows &  cols & rank & nullity & new \\ \midrule
   1  & 5         &  1  &  12  &  14  &   7  &    42 &    14 &   7  &   7 & 0 \\
   2  & 41        &  4  &  48  &  56  &  21  &   168 &    56 &  35  &  21 & 0 \\
   3  & 32        &  5  &  60  &  70  &  20  &   210 &    70 &  50  &  20 & 0 \\
   4  & 311       &  6  &  72  &  84  &  22  &   252 &    84 &  62  &  22 & 0 \\
   5  & 221       &  5  &  60  &  70  &  14  &   210 &    70 &  56  &  14 & 0 \\
   6  & 2111      &  4  &  48  &  56  &   9  &   168 &    56 &  47  &   9 & 0 \\
   7  & 11111     &  1  &  12  &  14  &   1  &    42 &    14 &  13  &   1 & 0
   \end{tabular}
   \end{center}
   \smallskip
   \caption{Degree 5: matrix ranks for all representations}
   \label{tableranksdegree5}
\bigskip
   \begin{center}
   \begin{tabular}{rlr|rrr|rrrr|r}
   & & & \multicolumn{3}{|c|}{lifted ($L_\lambda$)} & \multicolumn{4}{|c|}{all ($X_\lambda^t$)} \\
   \# & $\lambda$ & $d_\lambda$ & rows & cols & rank &  rows &  cols & rank & nullity & new \\ \midrule
   1  & 6         &  1  &  84  &  42  &  27  &   132 &    42 &  15  &  27 & 0 \\
   2  & 51        &  5  & 420  & 210  & 110  &   660 &   210 & 100  & 110 & 0 \\
   3  & 42        &  9  & 756  & 378  & 170  &  1188 &   378 & 208  & 170 & 0 \\
   4  & 411       & 10  & 840  & 420  & 176  &  1320 &   420 & 244  & 176 & 0 \\
   5  & 33        &  5  & 420  & 210  &  87  &   660 &   210 & 123  &  87 & 0 \\
   6  & 321       & 16  &1344  & 672  & 247  &  2112 &   672 & 425  & 247 & 0 \\
   7  & 3111      & 10  & 840  & 420  & 138  &  1320 &   420 & 282  & 138 & 0 \\
   8  & 222       &  5  & 420  & 210  &  67  &   660 &   210 & 143  &  67 & 0 \\
   9  & 2211      &  9  & 756  & 378  & 112  &  1188 &   378 & 266  & 112 & 0 \\
  10  & 21111     &  5  & 420  & 210  &  53  &   660 &   210 & 157  &  53 & 0 \\
  11  & 111111    &  1  &  84  &  42  &   8  &   132 &    42 &  34  &   8 & 0 \\
   \end{tabular}
   \end{center}
   \smallskip
   \caption{Degree 6: matrix ranks for all representations}
   \label{tableranksdegree6}
\bigskip
   \begin{center}
   \begin{tabular}{rlr|rrr|rrrr|r}
   & & & \multicolumn{3}{|c|}{lifted ($L_\lambda$)} & \multicolumn{4}{|c|}{all ($X_\lambda^t$)} \\
   \# & $\lambda$ & $d_\lambda$ & rows & cols & rank &  rows &  cols & rank & nullity & new \\ \midrule
    1 & 7         &  1  &  133 &  132 &   95 &   429  &  132  &   37 &   95 &   0 \\
    2 & 61        &  6  &  798 &  792 &  504 &  2574  &  792  &  288 &  504 &   0 \\
    3 & 52        & 14  & 1862 & 1848 & 1060 &  6006  & 1848  &  788 & 1060 &   0 \\
    4 & 511       & 15  & 1995 & 1980 & 1099 &  6435  & 1980  &  881 & 1099 &   0 \\
    5 & 43        & 14  & 1862 & 1848 &  992 &  6006  & 1848  &  856 &  992 &   0 \\
    6 & 421       & 35  & 4655 & 4620 & 2333 & 15015  & 4620  & 2287 & 2333 &   0 \\
    7 & 4111      & 20  & 2660 & 2640 & 1259 &  8580  & 2640  & 1381 & 1259 &   0 \\
    8 & 331       & 21  & 2793 & 2772 & 1333 &  9009  & 2772  & 1439 & 1333 &   0 \\
    9 & 322       & 21  & 2793 & 2772 & 1269 &  9009  & 2772  & 1503 & 1269 &   0 \\
   10 & 3211      & 35  & 4655 & 4620 & 2035 & 15015  & 4620  & 2585 & 2035 &   0 \\
   11 & 31111     & 15  & 1995 & 1980 &  800 &  6435  & 1980  & 1180 &  800 &   0 \\
   12 & 2221      & 14  & 1862 & 1848 &  751 &  6006  & 1848  & 1097 &  751 &   0 \\
   13 & 22111     & 14  & 1862 & 1848 &  705 &  6006  & 1848  & 1143 &  705 &   0 \\
   14 & 211111    &  6  &  798 &  792 &  269 &  2574  &  792  &  523 &  269 &   0 \\
   15 & 1111111   &  1  &  133 &  132 &   38 &   429  &  132  &   94 &   38 &   0
   \end{tabular}
   \end{center}
   \smallskip
   \caption{Degree 7: matrix ranks for all representations}
   \label{tableranksdegree7}
\end{table}

Before proceeding to degree 6, we make some comments.
If there are special identities in some degree $n \ge 5$,
then there will also be special identities in every degree $> n$.
To make this inference, we need to assume that the domain of the expansion
map is the free unital nonassociative algebra.
If $f( x_1, \dots, x_n )$ is a special identity in degree $n$,
then in the lifting $f( x_1 \cdot x_{n+1}, \dots, x_n )$ we can set $x_{n+1} = 1$,
showing that if the lifting is not special then neither is the original identity.
Thus, if we can prove that there are no special identities in degree 7,
it follows that there are also no special identities in degrees 5 or 6.
However, for completeness, we include the results for degree 6.

\subsection*{Degree 6}

For degree 6 we proceed analogously to the degree 5 case.
There are 12 lifted pre-Jordan identities in degree 5, each of which produces 7 liftings to degree 6,
for a total of 84 lifted identities in degree 6.
In degree 6, there are 42 binary association types, and 132 normal $D$-types.
Hence for each partition $\lambda$ of 6, with corresponding dimension $d_\lambda$,
the matrix $L_\lambda$ of lifted identities has size $84 d_\lambda \times 42 d_\lambda$,
and the expansion matrix $X_\lambda^t$ has size $132 d_\lambda \times 42 d_\lambda$.
In every case, we find that the rank of $L_\lambda$ equals the nullity of $X_\lambda^t$,
and so there are no new identities in degree 6; see Table~\ref{tableranksdegree6}.

At this point, the number of lifted identities is increasing rapidly,
and so the matrix $L_\lambda$ is becoming very large.
We can reduce the amount of memory required as follows.
Suppose that there are $k$ lifted identities in degree $n$.
Write $c_n$ for the number of binary association types in degree $n$.
Normally, the matrix $L_\lambda$ would have size $k d_\lambda \times c_n d_\lambda$.
However, if $k > c_n$, we can take $L_\lambda$ to have size $(c_n+1)d_\lambda \times c_n d_\lambda$.
We then process the lifted identities one at a time, storing their representation matrices
in the last row of $d_\lambda \times d_\lambda$ blocks, and compute the RCF of the matrix
after each lifted identity.
This saves memory but uses more time since we must perform many more row reductions.

\subsection*{Degree 7}

There are 84 lifted pre-Jordan identities in degree 6, each of which produces 8 liftings to degree 7,
for a total of 672 lifted identities in degree 7.
In degree 7, there are 132 binary association types.
Hence we take $L_\lambda$ to have size $133 d_\lambda \times 132 d_\lambda$.
In degree 7, there are 429 normal $D$-types, and so
the expansion matrix $X_\lambda^t$ has size $429 d_\lambda \times 132 d_\lambda$.
In every case, we find that the rank of $L_\lambda$ equals the nullity of $X_\lambda^t$,
and so there are no new identities in degree 7; see Table~\ref{tableranksdegree7}.

For degree 7, we kept track of which lifted identities increased the rank of $L_\lambda$
for at least one partition $\lambda$.
This reduced the number of pre-Jordan identities in degree 7 from 672 to 133.
Each of these has 9 liftings to degree 8, for a total of 1197 (much smaller than $9 \cdot 672 = 6048$).
In this way, we substantially speed up the computations in degree 8.


\section{Special identities in degree 8}

Owing to the size of the representation matrices in degree 8, and limits on available memory,
we had to split the problem into smaller parts and process each representation in a separate
Maple worksheet; this takes more CPU time but requires less memory.
In degree 8 we find new (and hence special) identities in partitions
$\lambda = 431$, 422, 332, 3311 and 3221, see Table \ref{tableranksdegree8}.

\begin{table}[h]
\small
   \begin{center}
   \begin{tabular}{rlr|rrr|rrrr|r}
   & & & \multicolumn{3}{|c|}{lifted ($L_\lambda$)} & \multicolumn{4}{|c|}{all ($X_\lambda^t$)} \\
   \# & $\lambda$ & $d_\lambda$ & rows & cols & rank &  rows &  cols & rank & nullity & new \\ \midrule
    1 & 8         &  1  &   479 &   429 &   339 &    479 &    429 &    90 &   339 &   0 \\
    2 & 71        &  7  &  3353 &  3003 &  2174 &   3353 &   3003 &   829 &  2174 &   0 \\
    3 & 62        & 20  &  9580 &  8580 &  5778 &   9580 &   8580 &  2802 &  5778 &   0 \\
    4 & 611       & 21  & 10059 &  9009 &  5939 &  10059 &   9009 &  3070 &  5939 &   0 \\
    5 & 53        & 28  & 13412 & 12012 &  7671 &  13412 &  12012 &  4341 &  7671 &   0 \\
    6 & 521       & 64  & 30656 & 27456 & 16930 &  30656 &  27456 & 10526 & 16930 &   0 \\
    7 & 5111      & 35  & 16765 & 15015 &  8951 &  16765 &  15015 &  6064 &  8951 &   0 \\
    8 & 44        & 14  &  6706 &  6006 &  3728 &   6706 &   6006 &  2278 &  3728 &   0 \\
    9 & 431       & 70  & 33530 & 30030 & 17721 &  33530 &  30030 & 12308 & 17722 &   1 \\
   10 & 422       & 56  & 26824 & 24024 & 13812 &  26824 &  24024 & 10211 & 13813 &   1 \\
   11 & 4211      & 90  & 43110 & 38610 & 21676 &  43110 &  38610 & 16934 & 21676 &   0 \\
   12 & 41111     & 35  & 16765 & 15015 &  8032 &  16765 &  15015 &  6983 &  8032 &   0 \\
   13 & 332       & 42  & 20118 & 18018 & 10039 &  20118 &  18018 &  7977 & 10041 &   2 \\
   14 & 3311      & 56  & 26824 & 24024 & 13056 &  26824 &  24024 & 10967 & 13057 &   1 \\
   15 & 3221      & 70  & 33530 & 30030 & 15853 &  33530 &  30030 & 14176 & 15854 &   1 \\
   16 & 32111     & 64  & 30656 & 27456 & 13956 &  30656 &  27456 & 13500 & 13956 &   0 \\
   17 & 311111    & 21  & 10059 &  9009 &  4289 &  10059 &   9009 &  4720 &  4289 &   0 \\
   18 & 2222      & 14  &  6706 &  6006 &  2978 &   6706 &   6006 &  3028 &  2978 &   0 \\
   19 & 22211     & 28  & 13412 & 12012 &  5803 &  13412 &  12012 &  6209 &  5803 &   0 \\
   20 & 221111    & 20  &  9580 &  8580 &  3929 &   9580 &   8580 &  4651 &  3929 &   0 \\
   21 & 2111111   &  7  &  3353 &  3003 &  1262 &   3353 &   3003 &  1741 &  1262 &   0 \\
   22 & 11111111  &  1  &   479 &   429 &   158 &    479 &    429 &   271 &   158 &   0
   \end{tabular}
   \end{center}
   \smallskip
   \caption{Degree 8: matrix ranks for all representations}
   \label{tableranksdegree8}
\end{table}

There are 429 binary association types in degree 8, and 1430 normal $D$-types.
To save memory, we processed the lifted identities 50 at a time,
so that $L_\lambda$ has size $( 429 + 50 ) d_\lambda \times 429 d_\lambda$.
We also used a similar strategy for the expansion matrix $X_\lambda^t$.
For each basic nonassociative monomial, we compute the expansion using the pre-Jordan product
and replace each term by its normalized form (as a linear combination of normal $D$-words).
We then split each expansion by considering the normal $D$-types 50 at a time;
in each iteration, we process those terms in the normalized expansion whose normal $D$-types
are among the current 50.
Hence $X_\lambda^t$ has size $( 429 + 50 ) d_\lambda \times 429 d_\lambda$,
the same as $L_\lambda$.
At the end of this computation, the ranks will be the same as if we had processed the expansions
all at once.

We do not give explicit expressions for the new special identities, since they involve far too many terms.

The column labelled `new' in Table \ref{tableranksdegree8} is identical to the corresponding column in
\cite[Table 7]{BremnerPeresi11} which gives the new identities in degree 8 for the Jordan diproduct in
an associative dialgebra.
A comparison of the results of this paper with those of \cite{BremnerPeresi11} 
shows that for $4 \le n \le 8$, the multiplicity of the irreducible $S_n$-module $[\lambda]$
in the space of new identities is the same for the pre-Jordan product in the free dendriform algebra
and the Jordan diproduct in the free associative dialgebra.
This result does not hold for $n = 3$, since the pre-Jordan product satisfies no
identities, but the Jordan diproduct satisfies right commutativity.
In the next section we state a conjecture which explains this fact in terms of Koszul duality of operads.


\section{Conjectures} \label{conjecturesection}

Our first conjecture extends the results of this paper and \cite{BremnerPeresi11}
to degree $n$.

\begin{conjecture}
Let $\dias$ and $\dend = \dias^!$ be the dual operads of associative dialgebras and
dendriform algebras respectively.

In degree $n \le 3$, every multilinear polynomial identity satisfied by the Jordan diproduct
in the free $\dias$-algebra is a consequence of right commutativity, $x(yz) - x(zy) \equiv 0$.
Let $P_n$ be the $S_n$-module of multilinear right-commutative polynomials in degree $n$
(the multilinear subspace in degree $n$ of the free right commutative algebra),
let $Q_n \subseteq P_n$ be the submodule of identities satisfied by the Jordan diproduct
in the free right-commutative algebra, and
let $R_n \subseteq Q_n$ be the submodule of identities which are consequences of identities
in lower degrees.

In degree $n \le 3$, there are no polynomial identities satisfied by the pre-Jordan product
in the free $\dend$-algebra.
Let $P'_n$ be the $S_n$-module of multilinear nonassociative polynomials in degree $n$
(the multilinear subspace in degree $n$ of the free nonassociative algebra),
let $Q'_n \subseteq P'_n$ be the submodule of identities satisfied by the pre-Jordan product
in the free dendriform algebra, and
let $R'_n \subseteq Q'_n$ be the submodule of identities which are consequences of identities
in lower degrees.

Then for all $n \ge 4$ we have an isomorphism of $S_n$-modules $Q_n / R_n \cong Q'_n / R'_n$.
\end{conjecture}

This conjecture suggests that a similar isomorphism might hold more generally
between the polynomial identities for corresponding bilinear operations in the free algebras
over any dual pair of binary quadratic non-$\Sigma$ operads.
(We recall from Loday \cite{Loday01} that a non-$\Sigma$ operad $\mathsf{O}$ is defined by the
conditions that the operations have no symmetry and the relations contain only the identity
permutation of the variables; this implies that each space $\mathsf{O}(n)$ is a direct sum of
copies of the regular representation of $S_n$.)
The duality theory for quadratic operads is well-developed, but we need to clarify the meaning
of the phrase `corresponding bilinear operations'.

The space of operations $\dias(2)$ is the tensor product of the vector space $V$ with basis
$\{ \dashv, \vdash \}$ with the group algebra $F S_2$; thus $\dias(2)$ has basis
  \[
  x \dashv y, \quad y \dashv x, \quad x \vdash y, \quad y \vdash x.
  \]
Similarly, the space of operations $\dend(2)$ is the tensor product of the vector space $W$
with basis $\{ \prec, \succ \}$ with the group algebra $F S_2$; thus $\dend(2)$ has basis
  \[
  x \prec y, \quad y \prec x, \quad x \succ y, \quad y \succ x.
  \]
By the duality between $\dias$ and $\dend$, we know that
$\dend(2) = \dias(2)^\ast \otimes_{FS_2} (\mathrm{sgn})$
where $(\mathrm{sgn})$ is the signature representation.

The Jordan diproduct in $\dias(2) = V \otimes FS_2$ is the element $x \dashv y + y \vdash x$.
The annihilator of this element in $\dias(2)^\ast$ is $( x \dashv y )^\ast - ( y \vdash x )^\ast$;
recall that $S_2$ acts only on the second factor in the tensor product.
The corresponding element in $\dend(2)$ requires twisting by the signs of the permutations,
giving $( x \dashv y )^\ast + ( y \vdash x )^\ast$,
and this represents the pre-Jordan product $x \succ y + y \prec x$.
(We identify $\dashv$ with $\succ$ and $\vdash$ with $\prec$,
opposite to Loday \cite[Proposition 8.3]{Loday01}.)
In other words, we can identify the spaces of operations in this dual pair of operads in a natural
way so that the pre-Jordan product corresponds to the Jordan diproduct.
For further details, see \cite[Proposition B.3]{Loday01}.

Now consider a variety $V$ of algebras with $k$ binary operations,
and assume that the operations have no symmetry.
The space $E$ of generating operations is isomorphic as an $S_2$-module
to the direct sum of $k$ copies of $F S_2$.
Therefore both $E^\ast$ and $E^\vee = E^\ast \otimes (\mathrm{sgn})$ are canonically isomorphic to $E$.
This means that for any bilinear operation $M$ in the variety $V$
(that is, any element $M$ of $E$),
there is a canonical corresponding bilinear operation $M^\vee$ in any variety
$V^\vee$ defined by the operations $E^\vee$.
In particular, we consider the case where the variety $V$ is defined by a quadratic operad,
and the variety $V^\vee$ is defined by the dual operad.

\begin{conjecture}
Let $\mathsf{O}$ and $\mathsf{O}^!$ be a dual pair of binary quadratic non-$\Sigma$ operads.
Let $\circ$ be a bilinear operation in $\mathsf{O}(2)$,
and let $\circ'$ be the corresponding operation in $\mathsf{O}^!(2)$.
Let $\mathcal{V}$ be the variety defined by the multilinear polynomial identities of degree $\le 3$
satisfied by the operation $\circ$ in the free $\mathsf{O}$-algebra,
and let $\mathcal{V}'$ be the variety defined by the multilinear polynomial identities of degree $\le 3$
satisfied by the operation $\circ'$ in the free $\mathsf{O}^!$-algebra.

Let $P_n$ be the $S_n$-module of multilinear $\mathcal{V}$-polynomials in degree $n$
(the multilinear subspace in degree $n$ of the free $\mathcal{V}$-algebra),
let $Q_n \subseteq P_n$ be the submodule of identities satisfied by the operation $\circ$
in the free $\mathsf{O}$-algebra, and
let $R_n \subseteq Q_n$ be the submodule of identities which are consequences of identities
in lower degrees.

Let $P'_n$ be the $S_n$-module of multilinear $\mathcal{V}'$-polynomials in degree $n$
(the multilinear subspace in degree $n$ of the free $\mathcal{V}'$-algebra),
let $Q'_n \subseteq P'_n$ be the submodule of identities satisfied by the operation $\circ'$
in the free $\mathsf{O}^!$-algebra, and
let $R'_n \subseteq Q'_n$ be the submodule of identities which are consequences of identities
in lower degrees.

Then for all $n \ge 4$ we have an isomorphism of $S_n$-modules $Q_n / R_n \cong Q'_n / R'_n$.
\end{conjecture}

A further generalization of this conjecture to operads which are not necessarily binary might
be possible, but we will not pursue this question here.


\section{Open problems}

Dendriform algebras can be obtained from associative algebras
by splitting the operation, and in the same way quadri-algebras can be obtained from dendriform
algebras.
This splitting procedure can be repeated any number of times, leading to a sequence of structures
with $2^n$ binary nonassociative operations, satisfying a coherent set of polynomial identities
in degree 3, which imply that the sum of the operations is associative.
These structures are often called Loday algebras \cite{Vallette08} or ABQR algebras
\cite{EbrahimiFardGuo05}; see also \cite{BaiBellierGuoNi12}.

As illustrated in Table \ref{BaiHouTable}, Lie and Jordan algebras can be obtained from bilinear
operations on associative algebras, and analogously, pre-Lie and pre-Jordan algebras can be
obtained from bilinear operations on dendriform algebras.
This process of obtaining nonassociative structures from associative structures can be extended
to obtain $L$-dendriform algebras \cite{BaiLiuNi10} and $J$-dendriform algebras \cite{BaiHou12} from
quadri-algebras.

An important open problem is to extend the Gr\"obner-Shirshov basis
for free dendriform algebras \cite{ChenWang10} to quadri-algebras and beyond.
This would allow us to extend the methods of the present paper to polynomial identities for
$J$-dendriform algebras.
It would also be very useful to have a theory of Gr\"obner bases for arbitrary ideals in the free
dendriform algebra.
In the dual setting of free associative dialgebras, this has been accomplished by Bokut et al.
\cite{BokutChenLiu10}.
This development is a necessary first step for the construction of universal enveloping dialgebras
(resp. dendriform algebras) for Leibniz algebras and Jordan dialgebras (resp. pre-Lie and pre-Jordan algebras).

Dual to quadri-algebras is the variety introduced by Vallette \cite[\S 5.6]{Vallette08}.
This variety can also be obtained by applying the KP algorithm
\cite{BremnerFelipeSanchezOrtega12,KolesnikovVoronin13}
to associative dialgebras.
It is an open problem to study analogues of Lie and Jordan structures in this setting.
This also raises the question of finding an analogue of the KP algorithm in the dendriform setting:
this algorithm takes a variety of algebras and produces the corresponding variety of dialgebras;
the analogous procedure in the dendriform setting would involve splitting of the operations
to produce the corresponding variety of dendriform algebras.
The BSO algorithm \cite{BremnerFelipeSanchezOrtega12,KolesnikovVoronin13} takes a multilinear $n$-ary
operation in an algebra and produces the corresponding operations in a dialgebra;
it would be useful to have an analogue of this process for dendriform algebras.

It would be natural to consider analogues of other nonassociative structures in the
settings of dendriform algebras, quadri-algebras, and the algebras defined by the dual operads.
For alternative structures, see \cite{Liu05,NiBai09}.
The same questions can be asked in the context of the associative trialgebras and dendriform trialgebras
introduced by Loday and Ronco \cite{LodayRonco04}.

As a final note, we recall that the operads defining Jordan algebras, Jordan dialgebras, and pre-Jordan
algebras, are all binary and cubic.
An important problem is to extend to cubic operads the theory of Koszul duality \cite{GinzburgKapranov}.


\section*{Acknowledgements}

We thank Pavel Kolesnikov for suggesting that we use computational linear algebra
to study relations between various classes of algebras in the ``dendriform world'',
and in particular, polynomial identities for the pre-Jordan product.
We also thank Jaydeep Chipalkatti for a helpful remark about representations of
symmetric groups.
Murray Bremner was supported by a Discovery Grant from NSERC, the Natural Sciences
and Engineering Research Council of Canada.
Sara Madariaga was supported by a Postdoctoral Fellowship from PIMS (Pacific Institute
for the Mathematical Sciences).



\begin{thebibliography}{99}

\bibitem{AguiarLoday04}
\textsc{M. Aguiar, J.-L. Loday}:
Quadri-algebras.
\emph{J. Pure Appl. Algebra}
191 (2004), no.~3, 205--221.

\bibitem{BaiBellierGuoNi12}
\textsc{C. Bai, O. Bellier, L. Guo, X. Ni}:
Splitting of operations, Manin products, and Rota-Baxter operators.
\textit{Int. Math. Res. Notices},
first published online January 22, 2012.
\texttt{doi:10.1093/imrn/rnr266 }

\bibitem{BaiHou12}
\textsc{C. Bai, D. P. Hou}:
J-dendriform algebras.
\emph{Front. Math. China}
7 (2012), no.~1, 29--49.

\bibitem{BaiLiuNi10}
\textsc{C. Bai, L. G. Liu, X. Ni}:
Some results on L-dendriform algebras.
\emph{J. Geom. Phys.}
60 (2010), no.~6-8, 940--950.

\bibitem{BokutChenHuang}
\textsc{L. A. Bokut, Y. Chen, J. Huang}:
Gr\"{o}bner-Shirshov Bases for $L$-algebras.
\texttt{arXiv: 1005.0118v1}

\bibitem{BokutChenLiu10}
\textsc{L. A. Bokut, Y. Chen, C. Liu}:
Gr\"obner-Shirshov bases for dialgebras.
\emph{Internat. J. Algebra Comput.}
20 (2010), no.~3, 391--415.

\bibitem{Bremner10}
\textsc{M. R. Bremner}:
On the definition of quasi-Jordan algebra.
\emph{Comm. Algebra}
38 (2010), no.~12, 4695--4704.

\bibitem{BremnerBook}
\textsc{M. R. Bremner}:
\emph{Lattice Basis Reduction: An Introduction to the LLL Algorithm and its Applications}.
Pure and Applied Mathematics, 300. 
CRC Press, Boca Raton, FL, 2012.

\bibitem{BremnerFelipeSanchezOrtega12}
\textsc{M. R. Bremner, R. Felipe, J. Sanchez-Ortega}:
Jordan triple disystems.
\emph{Comput. Math. Appl.}
63 (2012) 1039--1055.

\bibitem{BremnerMurakamiShestakov}
\textsc{M. R. Bremner, L. I. Murakami, I. P. Shestakov}:
Nonassociative algebras.
\emph{Handbook of Linear Algebra}.
Chapman \& Hall/CRC, Boca Raton, FL, 2007.

\bibitem{BremnerPeresi09b}
\textsc{M. R. Bremner, L. A. Peresi}:
An application of lattice basis reduction to polynomial identities for algebraic structures.
\emph{Linear Algebra Appl.}
430 (2009), no.~2-3, 642--659.

\bibitem{BremnerPeresi11}
\textsc{M. R. Bremner, L. A. Peresi}:
Special identities for quasi-Jordan algebras.
\emph{Comm. Algebra}
39 (2011), no.~7, 2313--2337.

\bibitem{Burde06}
\textsc{D. Burde}:
Left-symmetric algebras, or pre-Lie algebras in geometry and physics.
\emph{Cent. Eur. J. Math.}
4 (2006), no.~3, 323--357.

\bibitem{ChenWang10}
\textsc{Y. Chen, B. Wang}:
Gr\"{o}bner-Shirshov bases and Hilbert series of free dendriform algebras.
\emph{Southeast Asian Bull. Math.}
34 (2010), no.~4, 639--650.

\bibitem{Clifton}
\textsc{J. M. Clifton}:
A simplification of the computation of the natural representation of the symmetric group $S_n$.
\emph{Proc. Amer. Math. Soc.}
83 (1981), no.~2, 248--250.

\bibitem{EbrahimiFardGuo05}
\textsc{K. Ebrahimi-Fard, L. Guo}:
On products and duality of binary, quadratic, regular operads.
\emph{J. Pure Appl. Algebra}
200 (2005), no.~3, 293--317.

\bibitem{GinzburgKapranov}
\textsc{V. Ginzburg, M. M. Kapranov}:
Koszul duality for operads.
\emph{Duke Math. J.}
76 (1994), no.~1, 203--272.
Erratum: \emph{Duke Math. J.} 80 (1995), no.~1, 293.

\bibitem{Hentzel77}
\textsc{I. R. Hentzel}:
Processing identities by group representation.
\emph{Computers in Nonassociative Rings and Algebras}, pp.~13--40.
Academic Press, New York, 1977.

\bibitem{HouNiBai}
\textsc{D. P. Hou, X. Ni, C. M. Bai}:
Pre-Jordan algebras.
\emph{Math. Scand.}
(to appear)

\bibitem{KolesnikovVoronin13}
\textsc{P. S. Kolesnikov, V. Y. Voronin}:
On special identities for dialgebras.
\emph{Linear Multilinear Algebra}
61 (2013), no.~3, 377--391.

\bibitem{Kupershmidt99}
\textsc{B. A. Kupershmidt}:
What a classical $r$-matrix really is.
\emph{J. Nonlinear Math. Phys.}
6 (1999), no.~4, 448--488.

\bibitem{Liu05}
\textsc{D. Liu}:
Steinberg-Leibniz algebras and superalgebras.
\emph{J. Algebra}
283 (2005), no.~1, 199--221.

\bibitem{Loday01}
\textsc{J.-L. Loday}:
Dialgebras.
\emph{Dialgebras and Related Operads}, 7--66,
Lecture Notes in Math., 1763, Springer, Berlin, 2001.

\bibitem{LodayRonco04}
\textsc{J.-L. Loday, M. Ronco}:
Trialgebras and families of polytopes.
\emph{Homotopy theory: relations with algebraic geometry, group cohomology, and algebraic K-theory}, 369--398,
\emph{Contemp. Math.}, 346, Amer. Math. Soc., Providence, RI, 2004.

\bibitem{LodayVallette12}
\textsc{J.-L. Loday, B. Vallette}:
\emph{Algebraic Operads}.
Grundlehren der Mathematischen Wissenschaften, 346.
Springer, Heidelberg, 2012.

\bibitem{NiBai09}
\textsc{X. Ni, C. Bai}:
Pre-alternative algebras and pre-alternative bialgebras.
\texttt{arXiv:0907.3391}
(submitted 20 July 2009).

\bibitem{Vallette08}
\textsc{B. Vallette}:
Manin products, Koszul duality, Loday algebras and Deligne conjecture.
\emph{J. Reine Angew. Math.}
620 (2008) 105--164.

\bibitem{Vatne04}
\textsc{J. E. Vatne}:
\emph{Introduction to Operads}.
Available at: \texttt{http://www.uib.no/People/nmajv/}
(accessed 13 December 2012).

\bibitem{Zhelyabin00}
\textsc{V. N. Zhelyabin}:
On a class of Jordan $D$-bialgebras.
\emph{Algebra i Analiz}
11 (1999), no.~4, 64--94.

\end{thebibliography}
\end{document}